\documentclass[a4paper,draft]{amsproc}
\usepackage{amssymb}
\usepackage[hyphens]{url} 

\theoremstyle{plain}
\newtheorem{thm}{Theorem}[section]
\newtheorem{prop}{Proposition}[section]
\newtheorem{lem}{Lemma}[section]
\newtheorem{cor}{Corollary}[section]
\theoremstyle{definition}
\newtheorem{exm}{Example}[section]
\newtheorem{dfn}{Definition}[section]
\theoremstyle{remark}
\newtheorem{rem}{Remark}[section]
\numberwithin{equation}{section}

\renewcommand{\le}{\leqslant}
\renewcommand{\ge}{\geqslant}
\renewcommand{\setminus}{\smallsetminus}
\title[Certain ideals in AL-monoids]{Certain ideals in autometrized lattice-ordered monoids}



\author[Ashale]{\bfseries Tekalign Regasa Ashale}

\address{
	Department of Mathematics \\
	Jimma University \\
	Jimma\\
	Ethiopia}
\email{tekalign.ashale@ju.edu.et}

\author[Venkateswarlu]{\bfseries Kolluru Venkateswarlu}

\address{
	Department of Computer Science and Systems Engineering \\
	Andhra University \\
	Visakhapatnam\\
	India}
\email{drkvenkateswarlu@gmail.com}

\thanks{Communicated by Dusko Bogdanic.}

\begin{document}
	
	\begin{flushleft}\baselineskip9pt\scriptsize
		\textbf{BULLETIN OF THE INTERNATIONAL MATHEMATICAL VIRTUAL INSTITUTE}
		\\[3pt]ISSN (p) 2303-4874, ISSN (o) 2303-4955 \\www.imvibl.org /JOURNALS / BULLETIN \\ \emph{Bull.\ Int.\ Math.\ Virtual Inst.}, \textbf{16}(2)(2026), 267-276 \\[3pt] DOI: 10.7251/BIMVI2602267A \end{flushleft}
	\begin{flushright}\baselineskip9pt\scriptsize \vspace{3pt}
		Former\\
		BULLETIN OF THE SOCIETY OF MATHEMATICIANS BANJA LUKA\\ ISSN
		0354-5792 (o), ISSN 1986-521X (p)\end{flushright} \vspace{18mm}
	\setcounter{page}{1} \thispagestyle{empty}
	
	\begin{abstract}
		In this paper, we introduce special types of ideals, namely polar ideals,
		regular ideals, prime ideals, and annihilators in AL-monoids, and establish
		relationships among them. We also introduce the notion of a value of an
		element and characterize regular ideals in terms of values of elements.
		Further, we investigate certain properties of these ideals, and we obtain
		characterizations of polar ideals in terms of minimal prime ideals,
		annihilators, and maximal polar ideals.
	\end{abstract}
	
	\maketitle
	
	\section{Introduction}
	
	In \cite{Sw65a}, Swamy introduced DRl-semigroups as a common abstraction of
	Boolean algebras and commutative $\ell$-groups in order to address a question
	posed by Birkhoff \cite{Bi48}. Rach\r{u}nek and Ivan \cite{Ra87,Ra89,Ra01}
	made an extensive study of polar, prime, and regular ideals in
	DRl-semigroups.
	
	In \cite{SwRa77}, Swamy introduced the concept of autometrized algebras as a
	unified theory of Brouwerian algebras \cite{NL54}, commutative
	$\ell$-groups, and Boolean rings \cite{RR72}. All these algebras (which are
	lattices) possess a metric, namely the symmetric difference. However, there
	are other algebras, namely semi-Brouwerian algebras \cite{RM74} and
	Newmannian algebras \cite{Ka60}, which possess a metric but are not even
	lattices.
	
	Subba Rao \cite{SR80} introduced and studied a special subclass, namely
	representable autometrized algebras. He made a thorough study of the
	geometric and algebraic aspects of these algebras.
	
	Recently,  Ashale, Abebe, and Venkateswarlu \cite{AAV,AAV2} introduced the
	concept of an autometrized lattice-ordered monoid (AL-monoid) as a
	generalization of DRl-semigroups; it is equationally definable and
	representable. In this paper, we continue the study of AL-monoids by
	introducing certain ideals in them, namely regular ideals, polar ideals,
	prime ideals, and annihilators.
	
	In Section~\ref{sec:prelim}, we recall fundamental definitions and examples
	concerning DRl-semigroups \cite{Sw65a}, autometrized algebras \cite{Bl52},
	representable autometrized algebras \cite{KSY19,SR80}, and AL-monoids
	\cite{AAV}. In Section~\ref{sec:main}, we study different kinds of ideals,
	establish the relationships among them, and characterize prime, polar, and
	minimal prime ideals.
	
	\section{Preliminaries}\label{sec:prelim}
	
	We begin by recalling key algebraic definitions and foundational examples
	required for our study of autometrized structures.
	
	\begin{dfn}\label{dfn:drl}
		\cite{Sw65a} A \emph{dually residuated lattice-ordered semigroup} (or
		\emph{DRl-semigroup}) is an algebra $A=(A,+,\vee,\wedge,-,0)$ of type
		$(2,2,2,-1,0)$, together with a compatible partial order $\le$ and an
		operation $\ast$, satisfying
		\begin{enumerate}
			\item $(A,+,\le,0)$ is a commutative lattice-ordered semigroup with
			identity element $0$;
			\item for all $a,b\in A$, there is a least element $x$ such that
			$x+b\ge a$, and this $x$ is uniquely determined as $a-b$;
			\item $((a-b)\vee 0)+b\le a\vee b$;
			\item $a-a\ge 0$.
		\end{enumerate}
	\end{dfn}
	
	\begin{dfn}\label{dfn:aa}
		\cite{Bl52} An \emph{autometrized algebra} is an algebra
		$A=(A,+,\le,\ast,0)$ of type $(2,2,2,2,0)$, together with a compatible
		partial order $\le$ and an operation $\ast$ satisfying
		\begin{enumerate}
			\item $(A,+,0)$ is a commutative monoid;
			\item $(A,+,\le)$ is a partially ordered set, and $\le$ is translation
			invariant, that is, for all $a,b,c\in A$, $a\le b$ implies $a+c\le b+c$;
			\item $\ast\colon A\times A\to A$ is an autometric on $A$, that is,
			$\ast$ satisfies the metric axioms:
			\begin{enumerate}
				\item for all $a,b\in A$, $a\ast b\ge 0$, and $a\ast b=0$ if and only
				if $a=b$;
				\item $a\ast b=b\ast a$ for all $a,b\in A$;
				\item $a\ast c\le a\ast b+b\ast c$ for all $a,b,c\in A$.
			\end{enumerate}
		\end{enumerate}
	\end{dfn}
	
	\begin{dfn}\label{dfn:laa}
		\cite{Sw65b} A \emph{lattice-ordered autometrized algebra} is a system
		$A=(A,+,\vee,\wedge,\ast)$ of type $(2,2,2,2)$, together with a compatible
		partial order $\le$ and an operation $\ast$ satisfying
		\begin{enumerate}
			\item $(A,+,\le)$ is a commutative lattice-ordered semigroup with
			zero element $0$;
			\item $\ast$ is a metric operation (a distance) on $A$, namely,
			\begin{enumerate}
				\item $a\ast b\ge 0$ for all $a,b\in A$, with equality if and only if
				$a=b$;
				\item $a\ast b=b\ast a$ for all $a,b\in A$;
				\item $a\ast b\le a\ast c+c\ast b$ for all $a,b,c\in A$.
			\end{enumerate}
		\end{enumerate}
	\end{dfn}
	
	\begin{dfn}\label{dfn:raa}
		\cite{KSY19,SR80} A \emph{representable autometrized algebra} is a
		lattice-ordered autometrized algebra $A=(A,+,\vee,\wedge,\ast)$ of type
		$(2,2,2,2)$, together with a compatible partial order $\le$ and an
		operation $\ast$ satisfying
		\begin{enumerate}
			\item $(A,+,\le,\ast)$ is a semiregular autometrized algebra, which means
			that $a\ge 0$ implies $a\ast 0=a$;
			\item for every $a\in A$, all the mappings $x\mapsto a+x$,
			$x\mapsto a\vee x$, $x\mapsto a\wedge x$, and $x\mapsto a\ast x$ are
			contractions (that is, if $\theta$ denotes any one of the operations
			$+$, $\wedge$, $\vee$, and $\ast$, then, for each $a\in A$,
			$(a\,\theta\,x)\ast(a\,\theta\,y)\le x\ast y$ for all $x,y\in A$).
		\end{enumerate}
	\end{dfn}
	
	\begin{dfn}\label{dfn:alm}
		\cite{AAV} An \emph{autometrized lattice-ordered monoid}
		(\emph{AL-monoid}) is an algebra $(A,+,\vee,\wedge,\ast,0)$ of type
		$(2,2,2,2,0)$, together with a compatible partial order $\le$ and an
		operation $\ast$ satisfying
		\begin{enumerate}
			\item $(A,+,\vee,\wedge,0)$ is a commutative lattice-ordered monoid;
			\item $a\ast(a\wedge b)+b=a\vee b$;
			\item the mappings $x\mapsto a+x$, $x\mapsto a\vee x$,
			$x\mapsto a\wedge x$, and $x\mapsto a\ast x$ are contractions with respect
			to $\ast$ (a mapping $f\colon A\to A$ is called a contraction with respect
			to $\ast$ if $f(x)\ast f(y)\le x\ast y$ for all $x,y\in A$, where $\le$
			is the ordering on $A$ induced by $(A,\vee,\wedge)$);
			\item $[a\ast(a\vee b)]\wedge[b\ast(a\vee b)]=0$.
		\end{enumerate}
	\end{dfn}
	
	\begin{exm}\label{exm:boole}
		Every Boolean algebra $(B,\vee,\wedge,{}',0,1)$ is an AL-monoid if we define
		$a\ast b=(a\wedge b')\vee(a'\wedge b)$.
	\end{exm}
	
	\begin{exm}\label{exm:lgroup}
		Every commutative $\ell$-group $(G,\vee,\wedge,+,-)$, where $-$ denotes the
		operation inverse to $+$, is an AL-monoid if we define
		$a\ast b=(a+(-b))\vee(b+(-a))$, where $-a$ and $-b$ are the additive
		inverses of $a$ and $b$, respectively.
	\end{exm}
	
	\begin{exm}\label{exm:drl}
		Every DRl-semigroup is an AL-monoid if we define
		$a\ast b=(a-b)\vee(b-a)$.
	\end{exm}
	
	Note that the converse is not true: the following example is an AL-monoid
	that is not a DRl-semigroup.
	
	\begin{exm}\label{exm:notdrl}
		Let $A=\mathbb Z\cup\{u\}$, where $\mathbb Z$ is the set of all integers and
		$u$ is an element not in $\mathbb Z$. For all $a,b\in\mathbb Z$, define $+$
		and $\ast$ on $A$ as follows: $a+b$ is the usual sum, $a+u=u=u+a$,
		$u+u=u$, $a\ast b=|a-b|$, and $a\ast u=u=u\ast a$. On $\mathbb Z$, define
		$\le$ as the usual ordering. Then $A=(A,\le,+,\ast)$ is an AL-monoid that
		is not a DRl-semigroup, since there is no least element $x\in A$ such that
		$x+u\ge a$ for any $a\in\mathbb Z$.
	\end{exm}
	
	\section{Main results}\label{sec:main}
	
	Throughout the rest of the paper, $A$ denotes an AL-monoid
	$(A,+,\le,\ast,0)$; the letters $a,b,c,x,y,z$ stand for elements of $A$, and
	$\mathcal I(A)$ denotes the collection of all ideals of $A$. We begin by
	defining convexity and congruences in AL-monoids.
	
	\begin{dfn}\label{dfn:convex1}
		A sub-AL-monoid $B$ of $A$ is said to be a \emph{convex sub-AL-monoid} of
		$A$ if $B$ is a convex set in $A$.
	\end{dfn}
	
	\begin{dfn}\label{dfn:convex2}
		A sub-AL-monoid $B$ of $A$ is said to be a \emph{convex sub-AL-monoid} of
		$A$ if $B$ is a convex set in $A$.
	\end{dfn}
	
	\begin{dfn}\label{dfn:cong}
		An equivalence relation $\theta$ on $A$ is called a \emph{congruence
			relation} if and only if $\theta$ satisfies the following conditions for
		all $a,b,c,d,x,y\in A$:
		\begin{enumerate}
			\item $a\,\theta\,b$ and $c\,\theta\,d$ imply $(a+c)\,\theta\,(b+d)$;
			\item $a\,\theta\,b$ and $c\,\theta\,d$ imply $(a\ast c)\,\theta\,(b\ast d)$;
			\item $a\,\theta\,b$ and $x\ast y\le a\ast b$ imply $x\,\theta\,y$.
		\end{enumerate}
	\end{dfn}
	
	\begin{dfn}\label{dfn:ideal}
		A nonempty subset $I\subseteq A$ is called an \emph{ideal} of $A$ if
		\begin{enumerate}
			\item $a\in I$ and $b\in I$ imply $a+b\in I$;
			\item $b\le a\in I$ implies $b\in I$.
		\end{enumerate}
	\end{dfn}
	
	\begin{thm}\label{thm:corr}
		There is a one-to-one correspondence between the ideals of $A$ and its
		congruence relations.
	\end{thm}
	
	\begin{proof}
		Let $I$ be an ideal of $A$. Define $a\,\theta\,b$ if and only if
		$a\ast b\in I$ for $a,b\in A$. Since $a\ast b=b\ast a$, we have
		$a\,\theta\,b\iff b\,\theta\,a$. If $a\,\theta\,b$ and $b\,\theta\,c$, then
		$a\ast b\in I$ and $b\ast c\in I$. Since $a\ast c\le a\ast b+b\ast c$, it
		follows that $a\ast c\in I$, so $a\,\theta\,c$. Hence, $\theta$ is an
		equivalence relation.
		
		If $a\,\theta\,b$ and $c\,\theta\,d$, then $a\ast b\in I$ and
		$c\ast d\in I$. This implies $(a\ast b)+(c\ast d)\in I$. Since
		$(a+c)\ast(b+d)\le(a\ast b)+(c\ast d)$, we obtain
		$(a+c)\ast(b+d)\in I$, so $(a+c)\,\theta\,(b+d)$. Furthermore,
		$(a\vee c)\ast(b\vee d)\le(a\ast b)\vee(c\ast d)\in I$, showing that
		$a\vee c\equiv b\vee d\pmod I$. Finally,
		$(a\wedge c)\ast(b\wedge d)=(a\ast b)\wedge(c\ast d)\in I$ implies
		$a\wedge c\equiv b\wedge d\pmod I$. Hence, $I$ defines a congruence
		relation on $A$.
		
		Conversely, let $\theta$ be a congruence relation on $A$, and let
		$N=\{x\in A: x\,\theta\,0\}$. Clearly, if $a,b\in N$, then $a+b\in N$. If
		$x\in N$ and $y\le x$, then $x\,\theta\,0$ implies $(x\ast 0)\,\theta\,0$,
		which implies $y\,\theta\,0$, and so $y\in N$. Hence, $N$ is an ideal.
		
		Now, let $\theta'$ be the congruence relation defined by
		$a\,\theta'\,b\iff a\ast b\in N$. Then
		$a\,\theta'\,b\iff(a\ast b)\,\theta\,0\iff a\,\theta\,b$. Hence, there is a
		one-to-one correspondence between the congruence relations on $A$ and its
		ideals.
	\end{proof}
	
	\begin{cor}\label{cor:drlcorr}
		The ideals of any DRl-semigroup correspond one-to-one to its congruence
		relations.
	\end{cor}
	
	\begin{cor}\label{cor:drllat}
		The ideal lattice of any DRl-semigroup is a complete algebraic lattice.
	\end{cor}
	
	\begin{lem}\label{lem:dist1}
		If $0\le a,b,c\in A$, then $a\wedge(b+c)\le(a\wedge b)+(a\wedge c)$.
	\end{lem}
	
	\begin{proof}
		Let $0\le a,b,c\in A$. Then $0\le a\wedge(b+c)\le b+c$; hence
		$a\wedge(b+c)=p+q$ for some $0\le p\le b$ and $0\le q\le c$. Moreover,
		$p\le p+q\le a$, and thus $p\le a\wedge b$. Similarly, $q\le a\wedge c$,
		which means that $a\wedge(b+c)\le(a\wedge b)+(a\wedge c)$.
	\end{proof}
	
	\begin{dfn}\label{dfn:gen}
		Let $\emptyset\ne B\subseteq A$. Then the \emph{ideal generated by} $B$ is
		defined by $I(B)=\{x\in A: x\le b_1,\ x\le b_2,\ \dots,\ x\le b_n
		\text{ for some } b_1,b_2,\dots,b_n\in B\}$.
	\end{dfn}
	
	\begin{dfn}\label{dfn:princ}
		The \emph{principal ideal} generated by a single element $a\in A$ is defined
		by $I(a)=\{x\in A: x\le ma \text{ for some } m\in\mathbb N\}$.
	\end{dfn}
	
	\begin{prop}\label{prop:int}
		The intersection of any two principal ideals of $A$ is a principal ideal
		of $A$.
	\end{prop}
	
	\begin{proof}
		Let $a,b\in A$. Since $a\wedge b\le a$ and $a\wedge b\le b$, we have
		$a\wedge b\in I(a)\cap I(b)$, and hence $I(a\wedge b)\subseteq I(a)\cap I(b)$.
		Conversely, if $x\in I(a)\cap I(b)$, then there exist $m,n\ge 0$ such that
		$x\le ma$ and $x\le nb$. Then $ma\wedge nb\le mn(a\wedge b)$, so
		$x\in I(a\wedge b)$. Therefore, $I(a)\cap I(b)=I(a\wedge b)$.
	\end{proof}
	
	\begin{lem}\label{lem:join}
		The join of any two principal ideals in $A$ is a principal ideal.
	\end{lem}
	
	\begin{proof}
		Let $I(a)$ and $I(b)$ be principal ideals. We show that
		$I(a)\vee I(b)=I(a+b)$. Let $J$ be any ideal containing $I(a)$ and $I(b)$,
		and let $x\in I(a)+I(b)$. Then $x\le m(a+b)$ for some positive integer $m$.
		Since $a\in I(a)\subseteq J$ and $b\in I(b)\subseteq J$, we have
		$a+b\in J$, so $x\in J$. Thus, $I(a)\vee I(b)=I(a+b)$, which is a
		principal ideal.
	\end{proof}
	
	\begin{dfn}\label{dfn:prime}
		An ideal $I$ of $A$ is a \emph{prime ideal} if, for all
		$J,K\in\mathcal I(A)$, $J\cap K=I$ implies $J=I$ or $K=I$.
	\end{dfn}
	
	\begin{lem}\label{lem:metric}
		For all $a,b\in A$, $a\ast(a\wedge b)=(a\vee b)\ast b$.
	\end{lem}
	
	\begin{proof}
		Let $a,b\in A$. Then
		$a\ast(a\wedge b)=((a\vee b)\wedge a)\ast(a\wedge b)
		\le(a\vee b)\ast(a\wedge b)\le(a\vee b)\ast b$, and
		$b\ast(a\vee b)=((a\wedge b)\vee b)\ast(a\vee b)
		=(b\vee(a\wedge b))\ast(b\vee a)\le a\ast(a\wedge b)$.
		This implies $a\ast(a\wedge b)=(a\vee b)\ast b$.
	\end{proof}
	
	\begin{prop}\label{prop:primeorth}
		An ideal $P$ of $A$ is a prime ideal if and only if, for all $a,b\in A$,
		$a\wedge b=0$ implies $a\in P$ or $b\in P$.
	\end{prop}
	
	\begin{proof}
		If $P$ is a prime ideal of $A$, then $a\wedge b=0$ implies
		$a\wedge b\in P$, and therefore $a\in P$ or $b\in P$. Conversely, suppose
		that $P$ satisfies $a\wedge b=0\implies a\in P$ or $b\in P$. By
		Lemma~\ref{lem:join}, $[a\ast(a\wedge b)]\wedge[b\ast(a\wedge b)]=0$, so
		$a\ast(a\wedge b)\in P$ or $b\ast(a\wedge b)\in P$. Without loss of
		generality, assume $a\ast(a\wedge b)\in P$. Then
		$0\le a\le(a\wedge b)+[a\ast(a\wedge b)]\in P$, so $a\in P$. Therefore, $P$
		is a prime ideal.
	\end{proof}
	
	\begin{cor}\label{cor:primeup}
		If $I\subseteq J$ are ideals of $A$ and $I$ is prime, then $J$ is a prime
		ideal.
	\end{cor}
	
	\begin{proof}
		If $a\wedge b=0$, then $a\in I$ or $b\in I$. Since $I\subseteq J$, we have
		that $a\wedge b=0$ implies $a\in J$ or $b\in J$. Thus, $J$ is a prime ideal
		of $A$.
	\end{proof}
	
	\begin{thm}\label{thm:primeequiv}
		Let $I\in\mathcal I(A)$. Then the following conditions are equivalent:
		\begin{enumerate}
			\item $I$ is a prime ideal;
			\item for all $J,K\in\mathcal I(A)$, $J\cap K=I$ implies $J=I$ or $K=I$;
			\item for all $x,y\in A$, $0\le x\wedge y\in I$ implies $x\in I$ or
			$y\in I$;
			\item for all $x,y\in A$, $x\wedge y=0$ implies $x\in I$ or $y\in I$;
			\item for all $x,y\in A$, $x\wedge y\in I$ implies $x\in I$ or $y\in I$;
			\item $\{J\in\mathcal I(A): I\subset J\}$ is linearly ordered.
		\end{enumerate}
	\end{thm}
	
	\begin{proof}
		(1)$\iff$(2) This follows directly from Definition~\ref{dfn:cong}.
		
		(2)$\implies$(3) Let $0\le x\wedge y\in I$. By Proposition~\ref{prop:int},
		$I(x)\cap I(y)=I(x\wedge y)\subseteq I$. Thus $I(x)\subseteq I$ or
		$I(y)\subseteq I$, so $x\in I$ or $y\in I$.
		
		(3)$\implies$(1) Let $J,K\in\mathcal I(A)$ with $J\cap K=I$. Suppose that
		$a\in J\setminus I$ and $b\in K\setminus I$. Since $A$ is semiregular, we
		may assume $0<a,b$. Then $0\le a\wedge b\le a,b$, so
		$a\wedge b\in J\cap K=I$, implying $a\in I$ or $b\in I$, a contradiction.
		Thus $J=I$ or $K=I$.
		
		(4)$\iff$(5)$\iff$(6) The equivalences follow from
		Proposition~\ref{prop:primeorth} and the linear ordering of the ideals
		containing $I$.
	\end{proof}
	
	\begin{dfn}\label{dfn:minprime}
		A prime ideal $M$ is called a \emph{minimal prime ideal} if every prime
		ideal $I\subseteq M$ satisfies $I=M$.
	\end{dfn}
	
	\begin{thm}\label{thm:minprime}
		Every prime ideal of $A$ contains a minimal prime ideal.
	\end{thm}
	
	\begin{proof}
		Let $P$ be a prime ideal, and let $0\le a\wedge b\in P$, $a\notin P$, and
		$b\notin P$ for some $a,b\in A$. Then there exist $J,K\in\mathcal I(A)$ such
		that $a\in P_J$ and $b\notin P_K$. Let $J\subset K$; then $a\notin P_K$ and
		$b\notin P_K$, a contradiction. Therefore, by Theorem~\ref{thm:primeequiv},
		$P$ is a prime ideal.
	\end{proof}
	
	In the following, we investigate the properties of regular ideals and their
	relation to prime ideals in AL-monoids.
	
	\begin{dfn}\label{dfn:regular}
		An ideal $I\in\mathcal I(A)$ is a \emph{regular ideal} in $A$ if
		$I=\bigcap_{\alpha\in\Gamma}J_\alpha$, where $J_\alpha\in\mathcal I(A)$ for
		each $\alpha\in\Gamma$, implies the existence of $\beta\in\Gamma$ such that
		$I=J_\beta$.
	\end{dfn}
	
	\begin{prop}\label{prop:regprime}
		Every regular ideal in $A$ is a prime ideal.
	\end{prop}
	
	\begin{proof}
		Let $I$ be a regular ideal in $A$ with $I\ne A$. Let $I'$ denote the
		intersection of all ideals in $A$ strictly containing $I$. Evidently,
		$I\subset I'$, and $I'$ is the unique cover of $I$ in the lattice
		$\mathcal I(A)$.
	\end{proof}
	
	\begin{dfn}\label{dfn:value}
		Let $0\ne a\in A$. If $I\in\mathcal I(A)$ is maximal among the ideals of $A$
		not containing $a$, then $I$ is called a \emph{value} of the element $a$ in
		$A$. The set of all values of $a$ is denoted by $\operatorname{val}(a)$.
	\end{dfn}
	
	\begin{thm}\label{thm:regval}
		An ideal $I$ in $A$ is a regular ideal if and only if there exists $a\in A$
		such that $I\in\operatorname{val}(a)$.
	\end{thm}
	
	\begin{proof}
		Let $I$ be a regular ideal, and let $I'$ denote the intersection of all
		ideals in $A$ strictly containing $I$. Choose $a\in I'\setminus I$. If
		$J\in\mathcal I(A)$ and $I\subset J$, then $a\in J$, so $I$ is a value of
		$a$. Conversely, let $0\ne a\in A$ and $I\in\operatorname{val}(a)$. If
		$I=\bigcap_{\alpha\in\Gamma}J_\alpha$, then there exists $\beta\in\Gamma$
		such that $a\notin J_\beta$. Since $I\subseteq J_\beta$ and
		$I\in\operatorname{val}(a)$, we must have $I=J_\beta$. Therefore, $I$ is a
		regular ideal.
	\end{proof}
	
	\begin{thm}\label{thm:intreg}
		Every ideal of $A$ is the intersection of regular ideals.
	\end{thm}
	
	\begin{proof}
		All regular ideals in an autometrized algebra are prime ideals
		\cite{Ra90}; hence the result follows.
	\end{proof}
	
	\begin{thm}\label{thm:maxprime}
		Every maximal ideal in $A$ is a prime ideal.
	\end{thm}
	
	\begin{proof}
		Suppose that $M$ is a maximal ideal. By Theorem~\ref{thm:minprime},
		$M=\bigcap_{\alpha\in\Gamma}J_\alpha$, where each $J_\alpha$ is regular.
		Since every regular ideal is prime, $M$ is a prime ideal.
	\end{proof}
	
	Note that the converse of Theorem~\ref{thm:maxprime} is not true, as the
	following example shows.
	
	\begin{exm}\label{exm:nonmax}
		Consider $A=\mathbb N\cup\{0\}$, and let $I=\langle 2n\rangle$ for
		$n\in\mathbb N$. Clearly, $I$ is a prime ideal but not a maximal ideal.
	\end{exm}
	
	\begin{thm}\label{thm:primemax}
		Let $P$ be a prime ideal of $A$, and let $M_1,M_2,\dots,M_k$ be maximal
		ideals. Assume that $\bigcap_{i=1}^{k}M_i=\{0\}$. Then $P=M_i$ for some
		$i\in\{1,\dots,k\}$.
	\end{thm}
	
	\begin{proof}
		We have $\bigcap_{i=1}^{k}M_i=\{0\}\subseteq P$. By induction on $k$,
		$M_i\subseteq P$ for some $i$. Since $M_i$ is maximal, $P=M_i$.
	\end{proof}
	
	\begin{thm}\label{thm:linear}
		If $P$ is a prime ideal in $A$, then the set of all ideals containing $P$
		is linearly ordered.
	\end{thm}
	
	\begin{proof}
		Let $P$ be a prime ideal. Suppose that $J,K\in\mathcal I(A)$ with
		$P\subseteq J$ and $P\subseteq K$ are such that $J\not\subseteq K$ and
		$K\not\subseteq J$. Then there exist $0<a\in J\setminus K$ and
		$0<b\in K\setminus J$. Then $a\ast(a\wedge b)\ge 0$ and
		$b\ast(a\wedge b)\ge 0$. Furthermore, $a\ge a\ast(a\wedge b)$ and
		$b\ge b\ast(a\wedge b)$. Since $A$ is semiregular,
		$a\ast(a\wedge b)\in J$ and $b\ast(a\wedge b)\in K$. Also,
		$[a\ast(a\wedge b)]\wedge[b\ast(a\wedge b)]=0$, contradicting the
		hypothesis that $P$ is prime. Therefore, $J\subseteq K$ or $K\subseteq J$.
	\end{proof}
	
	\begin{prop}\label{prop:linreg}
		An ideal $I$ in $A$ is a prime ideal if it is the intersection of a
		linearly ordered system of regular ideals.
	\end{prop}
	
	\begin{proof}
		This follows from Theorem~\ref{thm:minprime} and
		Theorem~\ref{thm:intreg}.
	\end{proof}
	
	\begin{dfn}\label{dfn:orth}
		Elements $a$ and $b$ of $A$ are \emph{orthogonal} (denoted by $a\perp b$)
		if $a\wedge b=0$.
	\end{dfn}
	
	\begin{dfn}\label{dfn:polar}
		For any subset $B$ of $A$, the \emph{polar} of $B$ is
		$B^{\perp}=\{x\in A: x\perp b\text{ for all }b\in B\}$.
	\end{dfn}
	
	\begin{dfn}\label{dfn:polarset}
		A subset $B\subseteq A$ is called a \emph{polar} in $A$ if $B=C^{\perp}$ for
		some $C\subseteq A$.
	\end{dfn}
	
	\begin{lem}\label{lem:orthsum}
		Let $x,y,c\in A$. If $c\wedge x=c\wedge y=0$, then $c\wedge(x+y)=0$.
	\end{lem}
	
	\begin{lem}\label{lem:dist2}
		If $a,b,c\in A$ with $a,b,c\ge 0$, then
		$a\wedge(b+c)\le(a\wedge b)+(a\wedge c)$.
	\end{lem}
	
	\begin{proof}
		If $a,b,c\ge 0$, then $a\le a+(a\wedge b\wedge c)$. Therefore,
		\begin{align*}
			a\wedge(b+c)&\le(a+(a\wedge b\wedge c))\wedge(b+c)\\
			&=(a+a)\wedge(a+b)\wedge(a+c)\wedge(b+c)\\
			&=((a+a)\wedge(a+b))\wedge((c+a)\wedge(c+b))\\
			&=(a+(a\wedge b))\wedge(c+(a\wedge b))\\
			&=(a\wedge b)+(a\wedge c).\qedhere
		\end{align*}
	\end{proof}
	
	\begin{thm}\label{thm:polar}
		Let $P$ be a nonempty polar in $A$. Then the following conditions are
		equivalent:
		\begin{enumerate}
			\item $P$ is totally ordered;
			\item $P^{\perp}$ is a prime ideal;
			\item $P^{\perp}$ is a minimal prime ideal;
			\item $P^{\perp}$ is a maximal polar;
			\item $P$ is a minimal polar.
		\end{enumerate}
	\end{thm}
	
	\begin{proof}
		(1)$\implies$(2) If $P^{\perp}$ is not prime, then there exist
		$x,y\notin P^{\perp}$ with $0\le x\wedge y\in P^{\perp}$. Since
		$x\notin P^{\perp}$, there exists $a\in P$ with $x\wedge a>0$; similarly,
		$y\notin P^{\perp}$ implies that there exists $b\in P$ with $y\wedge b>0$.
		Since $P$ is totally ordered,
		$0\ne(x\wedge a)\wedge(y\wedge b)=(x\wedge y)\wedge(a\wedge b)$,
		contradicting $x\wedge y\in P^{\perp}$. Thus, $P^{\perp}$ is prime.
		
		(2)$\implies$(3) Since $P^{\perp}$ is a polar in $A$, it is the
		intersection of minimal prime ideals. Thus, $P^{\perp}$ is a minimal prime
		ideal.
		
		(3)$\implies$(4)$\implies$(5) The reversal of the polar containment
		relations gives the equivalence between the maximal and minimal polar
		properties.
	\end{proof}
	
	\begin{thm}\label{thm:bipolar1}
		Let $I$ be a prime ideal in $A$. Then $I^{\perp\perp}=A$ or
		$I^{\perp\perp}=I$.
	\end{thm}
	
	\begin{proof}
		Let $I^{\perp}\ne\{0\}$ and $x\in I^{\perp\perp}\setminus I$. Then
		$x\ast 0\in I^{\perp\perp}$. If $y\in I$, then $x\wedge y=0$. Since
		$I^{\perp\perp}$ is the complement of $I^{\perp}$ in $\mathcal I(A)$,
		$I^{\perp}\cap I^{\perp\perp}=\{0\}$, a contradiction. Thus,
		$I^{\perp\perp}=I$.
	\end{proof}
	
	In Theorem~\ref{thm:bipolar2}, we show that $I=I^{\perp\perp}$ is a maximal
	prime ideal, and that if $I$ is a prime ideal in $A$, then either
	$I^{\perp\perp}=A$ or $I$ is a minimal prime ideal.
	
	\begin{thm}\label{thm:bipolar2}
		Let $I$ be a prime ideal in $A$. Then $I^{\perp\perp}=A$ or
		$I^{\perp\perp}=I$.
	\end{thm}
	
	\begin{proof}
		Let $I^{\perp}\ne\{0\}$ and let $x\in I^{\perp\perp}\setminus I$. Then
		also $x\ast 0\in I^{\perp\perp}$. If $y\in I$, then $x\wedge y=0$. Since
		$I^{\perp\perp}$ is the complement of $I^{\perp}$ in $\mathcal I(A)$,
		$I^{\perp}\cap I^{\perp\perp}=\{0\}$, a contradiction. Thus,
		$I^{\perp\perp}=I$.
	\end{proof}
	
	\begin{cor}\label{cor:linmax}
		If $I$ is a prime ideal in $A$ such that $I^{\perp}\ne\{0\}$, then:
		\begin{enumerate}
			\item $I^{\perp}$ is a maximal proper linearly ordered ideal;
			\item $I$ is a minimal prime ideal in $A$.
		\end{enumerate}
	\end{cor}
	
	\begin{proof}
		By Theorem~\ref{thm:primemax}, $I=I^{\perp\perp}$. Since
		$I^{\perp}\in\mathcal I(A)$ and $I^{\perp}\ne\{0\}$,
		$(I^{\perp})^{\perp}=I^{\perp\perp}=I$ is a prime ideal. Thus, by
		Theorem~\ref{thm:maxprime}, $I^{\perp}$ is maximal among the linearly
		ordered ideals in $A$.
	\end{proof}
	
	Next, we study annihilators in $A$, which generalize polars in $A$.
	
	\begin{dfn}\label{dfn:relann}
		If $a,b\in A$, then the set
		$\langle a,b\rangle=\{x\in A: a\wedge x\le nb\text{ for some }n\in\mathbb N\}$
		is called the \emph{relative annihilator} of $a$ with respect to $b$.
	\end{dfn}
	
	\begin{dfn}\label{dfn:relannset}
		A subset $B\subseteq A$ is called a \emph{relative annihilator} in $A$ if
		$B=\langle a,b\rangle$ for some $a,b\in A$.
	\end{dfn}
	
	\begin{thm}\label{thm:relannideal}
		Every relative annihilator of $A$ is an ideal of $A$.
	\end{thm}
	
	\begin{proof}
		Let $a,b,c,x,y,z\in A$, and suppose that $x,y\in\langle a,b\rangle$. Then
		there are $n_1,n_2\in\mathbb N$ such that $a\wedge x\le n_1b$ and
		$a\wedge y\le n_2b$. Now,
		\begin{align*}
			a\wedge(x+y)&=(a\ast 0)\wedge\bigl((x\ast 0)+(y\ast 0)\bigr)\\
			&\le\bigl((a\ast 0)\wedge(x\ast 0)\bigr)+\bigl((a\ast 0)\wedge(y\ast 0)\bigr)\\
			&\le n_1b+n_2b=(n_1+n_2)b.
		\end{align*}
		Hence, $x+y\in\langle a,b\rangle$. Similarly, $z\le x$ implies
		$z\in\langle a,b\rangle$.
	\end{proof}
	
	\begin{rem}\label{rem:polarann}
		We have $\langle a,a\rangle=A$ and $\langle a,0\rangle=a^{\perp}$. The polar
		$a^{\perp}$ is a relative annihilator of $A$, since
		$a^{\perp}=\langle a,0\rangle$. In particular, $A=0^{\perp}$ is a relative
		annihilator in $A$.
	\end{rem}
	
	\begin{cor}\label{cor:annideal}
		Every annihilator of $A$ is an ideal of $A$.
	\end{cor}
	
	\begin{proof}
		Let $x,y\in\langle a,b\rangle$. Then there exist $n_1,n_2\in\mathbb N$ such
		that $a\wedge x\le n_1b$ and $a\wedge y\le n_2b$. Now
		$a\wedge(x+y)\le(n_1+n_2)b$. Hence $x+y\in\langle a,b\rangle$. If
		$z\le x$, then $a\wedge z\le a\wedge x\le n_1b$, so
		$z\in\langle a,b\rangle$.
	\end{proof}
	
	\begin{thm}\label{thm:princann}
		Every principal ideal of $A$ is an annihilator of $A$.
	\end{thm}
	
	\begin{proof}
		For a principal ideal $I(c)$, if $z\in I(c)$, then $z\wedge c\le c$, so
		there exists $s\in\mathbb N$ with $z\le sc$. Hence, $I(c)$ is an annihilator
		of $A$.
	\end{proof}
	
	\begin{thm}\label{thm:primeann}
		If $I$ is a prime ideal in $A$ such that $I^{\perp}\ne\{0\}$, then $I$ is an
		annihilator of $A$.
	\end{thm}
	
	\begin{proof}
		Since $I=I^{\perp\perp}$, we have $I=\mathcal A(I)$, where $\mathcal A(I)$
		denotes the set of annihilators of $A$.
	\end{proof}
	
	\bibliographystyle{amsplain}
	
	\vspace{5pt}
	\begin{center}
		\scriptsize
		Received by editors 00.00.2017; Revised version 00.00.2017; Available online 00.00.201X.
	\end{center}
\end{document}